\documentclass[11pt, reqno]{amsart}
\usepackage{amsmath, amsthm, a4, latexsym, amssymb}

\def\@rmrk#1#2{\refstepcounter
    {#1}\@ifnextchar[{\@yrmrk{#1}{#2}}{\@xrmrk{#1}{#2}}}

\makeatletter\@addtoreset{equation}{section}\makeatother

 \newfont{\bfit}{cmbxti10 scaled 2000}
 \newfont{\biggi}{cmr12 scaled 2000}

 \newcommand{\eps}{\varepsilon}

 \newcommand{\R}{\mathbb{R}}
 
 \newcommand{\N}{\mathbb{N}}

 \newcommand{\prob}{\mathbb{P}}

 \newcommand{\me}{\mathbb{E}}
 
 \renewcommand{\P}{\mathbb{P}}
 
 \newcommand{\skrib}{{\mathcal B}}
 \newcommand{\skric}{{\mathcal C}}
 \newcommand{\skrid}{{\mathcal D}}
 
 \newcommand{\skrie}{{\mathcal E}}
 
 \newcommand{\skrig}{{\mathcal G}}
 \newcommand{\skrih}{{\mathcal H}}

 \newcommand{\skrik}{{\mathcal K}}
 
 \newcommand{\skrim}{{\mathcal M}}
 \newcommand{\skriw}{{\mathcal W}}

 \newcommand{\skris}{{\mathcal S}}

 \newcommand{\skriy}{{\mathcal Y}}

 \newcommand{\sfrac}[2]{\mbox{$\frac{#1}{#2}$}}

\def\1{{\mathchoice {1\mskip-4mu\mathrm l}      
{1\mskip-4mu\mathrm l}
{1\mskip-4.5mu\mathrm l} {1\mskip-5mu\mathrm l}}}

\newcommand{\eq}{\begin{equation}}
\newcommand{\en}{\end{equation}}

\newenvironment{Proof}
{\vskip0.1cm\noindent{\bf Proof. }{\hspace*{0.3cm}}}%
{\nopagebreak {\hspace*{\fill}\rule{2mm}{2mm}}\\ }

{\nopagebreak {\hspace*{\fill}\rule{2mm}{2mm}}\\ }

\renewcommand{\subsection}{\secdef \subsct\sbsect}
\newcommand{\subsct}[2][default]{\refstepcounter{subsection}
\vspace{0.15cm}
{\flushleft\bf \arabic{section}.\arabic{subsection}~\bf #1  }
\nopagebreak\nopagebreak}
\newcommand{\sbsect}[1]{\vspace{0.1cm}\noindent
{\bf #1}\vspace{0.1cm}}

\newtheorem{theorem}{Theorem}[section]
\newtheorem{lemma}[theorem]{Lemma}
\newtheorem{cor}[theorem]{Corollary}

\newtheoremstyle{thm}{1.5ex}{1.5ex}{\itshape\rmfamily}{}
{\bfseries\rmfamily}{}{2ex}{}

\newtheoremstyle{rem}{1.3ex}{1.3ex}{\rmfamily}{}
{\itshape\rmfamily}{}{1.5ex}{}
\theoremstyle{rem}
\newtheorem{remark}{{\slshape\sffamily remark}}[]

\refstepcounter{subsection}

\def\thebibliography#1{\section*{References}
  \list%
  {\arabic{enumi}.}
    {\settowidth\labelwidth{[#1]}\leftmargin\labelwidth
    \advance\leftmargin\labelsep
    \parsep0pt\itemsep0pt
    \usecounter{enumi}}
    \def\newblock{\hskip .11em plus .33em minus .07em}
    \sloppy                   
    \sfcode`\.=1000\relax}

\begin{document}
\title[LLDP  and LDP  for Super-critical Communication Networks]
{\Large  Large  Deviations, Sharron-McMillan-Breiman Theorem   for  Super-Critical  Telecommunication  Networks}

\author[]{}

\maketitle
\thispagestyle{empty}
\vspace{-0.5cm}

\centerline{\sc By  E.  Sakyi-Yeboah$^1$, P. S. Andam$^1$,  L.  Asiedu$^1$ and  K.  Doku-Amponsah$^{1,2} $}

{$^1$Department  of  Statistics and Actuarial  Science, University of  Ghana,  BOX  LG 115, Legon,Accra}

{$^2$ Email: kdoku-amponsah@ug.edu.gh}

{$^2$ Telephone: +233205164254}
\vspace{0.5cm}

\renewcommand{\thefootnote}{}
\footnote{\textit{Acknowledgement: }  This  Research work has been  supported  by funds  from  the  Carnegie Banga-Africa Project,  University  of  Ghana}
\renewcommand{\thefootnote}{1}



\vspace{-0.5cm}

\begin{quote}{\small }{\bf Abstract.}
In  this  article  we  obtain  large  deviation  asymptotics  for  supercritical communication  networks  modelled  as   signal-interference-noise  ratio  networks. To  do  this,  we  define  the  empirical  power  measure  and  the  empirical  connectivity  measure,  and  prove   joint large  deviation  principles(LDPs)  for the    two  empirical  measures on  two  different  scales i.e. $\lambda$  and  $\lambda^2 a_{\lambda},$  where   $\lambda$  is  the intensity  measure  of  the   poisson  point  process (PPP)  which defines   the  SINR  random network.Using  this  joint LDPs  we  prove  an  asymptotic  equipartition  property   for  the  stochastic  telecommunication  Networks  modelled  as  the  SINR  networks. Further,  we  prove  a   Local  large  deviation  principle(LLDP)  for  the  SINR  Network.  From  the  LLDP  we  prove  the  a  large  deviation  principle,  and  a  classical  McMillian  Theorem   for  the stochastic  SNIR  network  processes. Note, for  tupical  empirical  connectivity measure, $q\pi\otimes\pi,$ we  can  deduce  from  the  LLDP    a  bound  on  the   cardinality  of  the  space of  SINR  networks  to  be  approximately  equal  to $\displaystyle e^{\lambda^2 a_{\lambda}\|q\pi\otimes\pi\|H\big(q\pi\otimes\pi/\|q\pi\otimes\pi\|\big)},$
where  the  connectivity probability  of  the  network,  $Q^{z^\lambda} ,$  satisfies  $ a_{\lambda}^{-1}Q^{z^\lambda} \to q.$  Observe,  the  LDP  for the  empirical measures  of the  stochastic  SINR  network  were obtained    on  spaces  of  measures  equipped  with  the  $\tau-$  topology, and  the  LLDPs  were obtained in  the  space  of  SINR  network  process  without  any  topological  restrictions.

\end{quote}\vspace{0.5cm}

\textit{Keywords: } Super-critical sinr networks, Poisson  Point  Process, Empirical power measure, Empirical connectivity measure,  Large deviations, Relative  entropy, Entropy

\vspace{0.3cm}

\textit{AMS Subject Classification:} 60F10, 05C80, 68Q87
\vspace{0.3cm}


\section{Introduction   and Background}
\subsection{Introduction}.
Large  deviations may  be  regarded  as  a  group  of  mathematical  techniques (stochastic  methods) often  use  to  estimate  asymptotic  properties  of  increasingly  rare  events  such  as their  empirical measures and  most  likely  manner  of  occurrence.  See, for  example,  \cite{We1995}.  There  are  many  applications  of  large deviation  techniques  to  SINR  networks as   a model  for  Telecommunication  networks. Some  of this  applications  include  but  not  limited  to  the  analysis  of  bi-stability  in  networks, example  notorious  bi-stability  in  multiple  access  protocols  such  as  the  Aloha,  and the  stochastic  behaviour  of  ATM  such as the  admission control,  sizing  of  internal  buffers,  and  the  simulation  of  ATM models. See,\cite{We1995}.\\

The  Shanno-MacMillian-Breiman (SMB)  Theorem  or  tte  asymptotic  equipartition  property   may  be  regarded  as  the  strong  law  of  large  numbers in  information  theory.  It says  output  source  of  a  stochastic  data  source  may  be  partition into  two  sets, namely  the  set  of  typical  events  and  the  set  of  atypical  events. The  SMB  is  the foundation  of  all  approximate  pattern  matching  and  coding  algorithms.

Researchers over the  last  two  decades have  given  some  large  deviation  analysis  for   telecommnunication networks  modelled  as a  sequence  of  i.i.d random  variables  and or   markov  chains in  discrete and  continuous  times.  See, \cite{We1995}  and reference  therein. \cite{SAD2020}  and  \cite{SKAD2020}  defined  empirical  measures  on  the  SINR  network  and  proved some  jonit  LDP  results   including  the  SMB  and  the classical  MacMillian  theorem  for the  dense  or critical   telecommunication  networks  modelled  as  the  SINR  network.\\

In  this  article  we  prove   joint  large  deviation  principles  on  the  scales  $\lambda$  and  $\lambda^2a_{\lambda}$,  where  $\lambda$ is the  intensity  measure  of  the  underlining PPP  of  the  SINR  network. See, \cite{DA2006}  or  \cite{DM2010}    for  similar  results  fore  the  colored  random  graph  models.From  these  LDPs  we  prove  an  asymptotic  equipartition  property, see example  \cite{DA2012},  for  the  SINR  networks.\\

Further,  we  prove  a  local  LDP for the  SINR  networks. See example, \cite{DA2017a} or \cite{DA2017c} and reference therein. From  the  local  LDP  we deduce  asymptotic  bounds  on  the  cardinality  of  the  set  of  SINR  networks  for  a  given  typical   empirical  power  measure. We  also  prove  from  the local  LDP  an  LDP for  the  SINR  network  processes.

The  remaining  part  of  the  article  is  organized  in  this  manner:  Section~\ref{Sec2} contains  the  main  results; Theorem~\ref{main1a}, Theorem~\ref{main1a-1} Theorem~\ref{main1b}, Theorem~\ref{main1c}, Corollary~\ref{cardinality}  and  Corollary~\ref{main2d}. In  Section~\ref{Sec3}  the  main  results of  the  article, Theorem~\ref{main1a}.  Section~\ref{Sec4} contain  proof  of  the  SBM,  see  Theorem~\ref{main1b}  and  Section~\ref{Sec5}; Proof  of  Theorem~\ref{main1c}, Corollary~\ref{cardinality} and  Corollary~\ref{main2d}. Finally.  we  give  the conclusion  of  the  article  in  Section~\ref{Sec6}

\subsection{Background}\label{SINR}\label{Sec1}

We fix  dimension  $d\in\N$    and some  measureable  set  $\skrid\subset \R^d$    with  respect  to  the  Borel-Sgma  algebra  $\skrib(\R^d).$ 
For an  intensity function,  $\lambda \pi:\skrid \to [0,1]$,  a transition kernel from  $\skrid$   to  $(0,\,\infty),$  $\skrik$ and      a  path loss  function, $\beta(\ell)=\ell^{-r}, $   where  $r \in(0,\infty),$ and   
  some technical constants; $\tau^{(\lambda)}, \gamma^{(\lambda)}:(0\,,\,\infty)\to (0\,,\,\infty),$ we  define  the  Sinr network  model  as  follows:

\begin{itemize}
\item  We pick  $Z=(Z_i)_{i\in I}$   a  Poisson  Point  process  (PPP)  with rate measure  $\lambda \pi:D \to [0,1]$.
\item Given  $Z,$   we    assign  each  $Z_i  $   a  power  $\eta(Z_i)=\eta_i$  independently  according  to  the  transition  function   $\skrik(\cdot \,,\,Z_i).$
\item For   any  two  powered  points  $((Z_i,\eta_i),(Z_j,\eta_j))$  we    connect  a  link  iff  $$SINR(Z_i,Z_j, Z) \ge \tau^{(\lambda)}(\eta_j) \mbox{  and      $SINR(Z_j,Z_i, Z)\ge \tau^{(\lambda)}(\eta_i),$}$$  where $$SINR(Z_j,Z_i, Z)=\frac{\eta_i \beta(\|Z_i-Z_j\|)}{N_0 +\gamma^{(\lambda)}(\eta_j)\sum_{i\in I\setminus\{j\} }\eta_i\beta(\|Z_i-Z_j\|)}$$
\end{itemize}

We shall  consider  $Z^{\lambda}:=Z^{\lambda}(\eta, \skrik, \beta)=\Big\{[(Z_i,\eta_i), j\in I], \, E\Big\}$  under  the  joint  law  of  the  powered  Poisson Point  Process  and    the  Network.  We  will  interpret   $Z^{\lambda}$    as   an SINR  Network  and   $ (Z_i,\eta_i):= Z_i^{\lambda}$   as  the  power type  of    device  $i.$  We  recall   from  \cite{SAD2020} that
	the link/connectivity  probability  of  the SINR  network,   $ Q^{z^{\lambda}}$,  is  given  by  $ Q^{z^{\lambda}}((x,\eta_x),(y,\eta_y))= e^{-\lambda q_{\lambda}^{\skrid} ((x,\eta_x),(y,\eta_y))) },$  where
	
	$$q_{\lambda}^{\skrid} ((x,\eta_x),(y,\eta_y))= \int_{D} \Big[\sfrac{ \tau^{(\lambda)}(\eta_x) \gamma^{(\lambda)}(\eta_x)  }{\tau^{(\lambda)}(\eta_x) \gamma^{(\lambda)}(\eta_x)+(\|z\|^{\ell}/\|x-y \|^{\ell})} + \sfrac{ \tau^{(\lambda)}(\eta_y) \gamma^{(\lambda)}(\eta_y)  }{\tau^{(\lambda)}(\eta_y) \gamma^{(\lambda)}(\eta_y)+(\|z\|^{\ell}/\|y-x \|^{\ell})}\Big] \pi(dz).$$

We  have  assumed there exists  a  sequence  of  real  numbers,$a_{\lambda}$   and  a function  $q :\skrid\times \R_+\to (0,\infty)$ such that  $\lambda^{2}a_{\lambda}\to\infty$  and  $\displaystyle \lim_{\lambda\uparrow\infty}a_{\lambda}^{-1}Q^{z^{\lambda}}((x,\eta_x),(y,\eta_y))=q((x,\eta_x),(y,\eta_y)).$

Sakyi-Yeboah et. al~\cite{SKAD2020} studied  the  critical SINR Networks (i.e.  $\lambda a_{\lambda}\to 1$). In  this paper  we  shall  look  at  Sup-critical SINR Networks.( i.e.$\lim_{\lambda\to\infty}\lambda a_{\lambda}\to \infty$).\\

 We  define  the  set $\skris(\skrid)$  by

 \begin{equation}
\skris(\skrid)=\cup_{x\subset \skrid}\Big\{x:\,\, |x\cap A|<\infty\,\, ,\mbox{for\, any  bounded  $A\subset \skrid$ }\Big \}.
\end{equation}

Write  $\skriw= \skris(\skrid\times\R_+)$   and    $\skrim(\skriw)$,  denote  the  space  of  positive  measures  on  the  space  $\skriw$   equipped  with  $\tau-$  topology. Note, $\skriw$ a locally  finite  subset  of  the  set  $\skriw.$  See, example,  \cite{SKAD2020}  or  \cite{JK2018}
   For any SINR Network  $Z^\lambda$  we  define a probability measure, the
\emph{empirical power measure}, ~ $U_1^{\lambda}\in\skrim(\skriw)$,~by
$$U_1^{\lambda}\big((x,\eta_x)\big ):=\frac{1}{\lambda}\sum_{i\in I}\delta_{Z_i^{\lambda}}\big((x,\eta_x)\big)$$
and a  finite measure, the \emph{empirical connectivity measure}
$U_2^{\lambda}\in\skrim(\skriw\times \skriw),$ by
$$U_2^{\lambda}\big((x,\eta_x),(y,\eta_y)\big):=\frac{1}{\lambda^2a_{\lambda}}\sum_{(i,j)\in E}[\delta_{(Z_i^{\lambda},Z_j^{\lambda})}+
\delta_{(Z_j^{\lambda},Z_i^{\lambda})}]\big((x,\eta_x),(y,\eta_y)\big).$$

 Note  that the  total mass  $\|U_1^{\lambda}\|$ of  the  empirical power  measure  is $\1$  and  total  mass  of  
the empirical link measure is
$2|E|/\lambda^2a_{\lambda}$.

\section { Main  Results}\label{mainresults}\label{Sec2}

Theorem~\ref{main1a},  is  a  Joint  Large  deviation principle   for  the  empirical  measures  of  the Sinr  network models.We  recall  from  Subsection~\ref{Sec1} the  definition  of  $q_{\lambda}^{\skrid}$  as  

  $$q_{\lambda}^{\skrid} ((x,\eta_x),(y,\eta_y))= \int_{D} \Big[\sfrac{ \tau^{(\lambda)}(\eta_x) \gamma^{(\lambda)}(\eta_x)  }{\tau(\eta_x) \gamma(\eta_x)+(\|z\|^{\ell}/\|x-y \|^{\ell})} + \sfrac{ \tau^{(\lambda)}(\eta_y) \gamma^{(\lambda)}(\eta_y)  }{\tau^{(\lambda)}(\eta_y) \gamma^{(\lambda)}(\eta_y)+(\|z\|^{\ell}/\|y-x \|^{\ell})}\Big] \eta(dz)$$    and   write	 $$q\pi\otimes\pi((x,\eta_x),(y,\eta_y))):=q((x,\eta_x),(y,\eta_y))\mu((x,\eta_x))\mu((y,\eta_y)).$$  

\begin{theorem}	\label{main1a}
	Let   $Z^{\lambda}$  is  a super critical powered Sinr  network  with rate measure
	$\lambda \eta:D \to [0,1]$ and   a   power  probability  function  $\skrik(\cdot,\eta)=ce^{-c\eta},$  $\eta\ge 0$  and  path  loss  function   $\beta(r)=r^{-\ell}, $  for  $\ell>0.$   Thus, the connectivity  probability  $Q^{z^\lambda} $  of   $Z^{\lambda}$  satisfies  $a_{\lambda}^{-1}Q^{z^\lambda}\to q$  and  $\lambda a_{\lambda} \to \infty.$  Then, as  $\lambda\to \infty$,  the  pair of  measures  $(U_1^{\lambda},U_2^{\lambda})$  satisfies a  large  deviation principle  in the  space 
		$\skrim(\skriw)\times \skrim(\skriw\times\skriw)$
		
		\begin{itemize}
	\item [(i)]	with  speed   $\lambda$  and  good rate function

		\begin{equation}
		\begin{aligned}
			I_{Sc}^{1}\big(\pi,\nu\big)= \left\{\begin{array}{ll}H\Big(\pi\Big |\eta\otimes q\Big)&\,\,\mbox{ if $ \nu=q\pi\otimes\pi $ }\\
		\infty & \mbox{otherwise.}
		\end{array}\right.
		\end{aligned}
		\end{equation}
		
		\item[(ii)]with  speed   $\lambda^2a_{\lambda}$  and  good rate function

			\begin{equation}
		I_{Sc}^{2}\big(\pi,\nu\big)= \frac{1}{2} \skrih\Big(\nu\|q\pi\otimes\pi\Big)
		\end{equation}

		\end{itemize}
		
	where

\begin{equation}
\begin{aligned}
	\skrih(\nu\|q\pi\otimes\pi):= \left\{\begin{array}{ll} H(\nu\,\|\,q\pi\otimes \pi)+\Big(\|q\pi \otimes\pi\|-\|\nu\|\Big),  & \mbox{if  $\|\nu\|>0.$  }\\
		\infty & \mbox{otherwise.}		
	\end{array}\right.
\end{aligned}
\end{equation}

and $$q\pi\otimes\pi((x,\eta_x),(y,\eta_y)))=q((x,\eta_x),(y,\eta_y))\pi((x,\eta_x))\pi((y,\eta_y)).$$
\end{theorem}

Theorem~\ref{main1a-1}  below  is  a  key  step  in   the proof  of   Theorem~\ref{main1a}.  See subsection~\ref{Sec4}.  
\begin{theorem}\label{main1a-1}
	Let   $Z^{\lambda}$  is  a super critical powered Sinr  network  with rate measure
	$\lambda \eta:D \to [0,1]$ and   a   power  probability  function  $\skrik(\cdot,\eta)=ce^{-c\eta},$  $\eta\ge 0$  and  path  loss  function   $\beta(r)=r^{-\ell}, $  for  $\ell>0.$   Thus, the connectivity  probability  $Q^{z^\lambda} $  of   $Z^{\lambda}$  satisfies  $a_{\lambda}^{-1}Q^{z^\lambda}\to q$  and  $\lambda a_{\lambda} \to \infty.$  Let  $Z^{\lambda}$  be  a super critical powered Sinr  network  conditional  on  event  $\big\{ U_1^{\lambda}=\pi, \big\}$. Then, as  $\lambda\to \infty$,  the  pair of  measures  $U_2^{\lambda}$  satisfies a  large  deviation principle  in the  space 
	$\skrim(\skriw)\times \skrim(\skriw\times\skriw)$
	
	\begin{itemize}
		\item [(i)]	with  speed   $\lambda$  and  good rate function

		\begin{equation}
		\begin{aligned}
		I_{\pi}^{1}\big(\nu\big)= \left\{\begin{array}{ll} 0&\,\,\mbox{ if $ \nu=q\pi\otimes\pi $ }\\
		\infty & \mbox{otherwise.}
		\end{array}\right.
		\end{aligned}
		\end{equation}
		
		\item[(ii)]with  speed   $\lambda^2a_{\lambda}$  and  good rate function

		\begin{equation}
		I_{\pi}^{2}\big(\nu\big)= \frac{1}{2} \skrih\Big(\nu\|q\pi\otimes\pi\Big).
		\end{equation}

	\end{itemize}

	\end{theorem}

\begin{theorem}	\label{main1b}
	Let   $Z^{\lambda}$  is  a super critical powered Sinr  network  with rate measure
	$\lambda \mu:D \to [0,1]$ and   a   power  probability  function  $\skrik(\eta)=ce^{-c\eta}, \eta>0$   and  path  loss  function   $\beta(r)=r^{-\ell}, $  for  $\ell>.$  Thus, the connectivity  probability  $Q^{z^\lambda} $  of   $Z^{\lambda}$  satisfies  $a_{\lambda}^{-1}Q^{z^\lambda}\to q$  and  $\lambda a_{\lambda} \to \infty.$   Suppose  the  sequence $a_{\lambda}$  of  $Z^{\lambda}$  is  such  that  $\lambda a_{\lambda}\,\log\lambda\to\infty$  and  $a_{\lambda}/\log\lambda\to-1.$	Then,  we  have  
	$$\lim_{\lambda\to\infty}\prob\Big\{\Big|-\frac{1}{a_{\lambda}\lambda^2\log\lambda}\log P(Z^{\lambda})-\int_{\skriw\times\skriw}q((x,\eta_x),(y,\eta_y))q(d\eta_x)q(d\eta_y)dxdy\Big|\ge \eps\Big\}=0.$$
	
\end{theorem}

\begin{theorem}	\label{main1c}
	Let   $Z^{\lambda}$  is  a super critical powered Sinr  network  with rate measure
$\lambda \mu:D \to [0,1]$ and   a   power  probability  function  $\skrik(\eta)=ce^{-c\eta}, \eta>0$   and  path  loss  function   $\beta(r)=r^{-\ell}, $  for  $\ell>.$  Thus, the connectivity  probability  $Q^{z^\lambda} $  of   $Z^{\lambda}$  satisfies  $a_{\lambda}^{-1}Q^{z^\lambda}\to q$  and  $\lambda a_{\lambda} \to \infty.$
Then, 
	\begin{itemize}
		\item  for  any  functional  $\nu\in\skrim_{\pi}$  and    a  number $\eps>0$,  there  exists  a  weak  neighbourhood $B_{\nu}$  such  that    
		$$\P_{\pi}\Big\{Z^{\lambda}\in \skrig_P\,\Big|\, L_2^{\lambda}\in B_{\nu}\Big\}\le e ^{-\sfrac{1}{2}\lambda^2 a_{\lambda} 	\skrih(\nu\|q\pi\otimes\pi)-\lambda a_{\lambda}\eps}.$$
		\item  for  any    $\nu\in \skrim_{\pi}$,  a  number  $\eps>o$  and  a  fine  neighbourhood $B_{\nu} $,  we  have  the  estimate: 
		$$\P_{\pi}\Big\{Z^{\lambda}\in \skrig_P\,\Big|\, L_2^{\lambda}\in B_{\nu}\Big\}\ge e^{-\sfrac{1}{2}\lambda ^2a_{\lambda}	\skrih(\nu\|q\pi\otimes\pi)+\lambda_{\lambda} a_{\lambda}\eps}.$$
	\end{itemize}

	\end{theorem}

We  define  for  telecommunication  networks  an  entropy  $h: \skrim(\skriw\times\skriw)\to [0.\infty]$   by

\begin{equation}\label{equ3}
h(\nu):=\Big(\|\nu\|-\|\lambda\pi\otimes\pi\|-\Big\langle \nu\, ,\,\log \sfrac{\nu}{\|q\pi\otimes\pi\|}\Big\rangle\Big)/2.
\end{equation}

\begin{cor}[McMillian Theorem]\label{cardinality}\label{main2c} Let $\skrig_p$ be a super critical powered Sinr  network  with rate measure
	$\lambda \mu:D \to [0,1]$ and   a   power  probability  function  $\skrik(\eta)=ce^{-c\eta}, \eta>0$   and  path  loss  function   $\beta(r)=r^{-\ell}, $  for  $\ell>0.$  Thus, the connectivity  probability  $Q^{z^\lambda} $  of every  $z^{\lambda}\in \skrig_p$  satisfies  $a_{\lambda}^{-1}Q^{z^\lambda}\to q$  and  $\lambda a_{\lambda} \to \infty.$

	\begin{itemize}
		
		\item[(i)] For  any empirical connectivity  measure  $\nu$   on  $\skriw\times\skriw$   and  $\eps>0,$  there  exists  a neighborhood  $B_{\nu}$  such  that
		$$ Card\Big(\big\{z^{\lambda}\in\skrig_p\,| \,L_2^{\lambda}\in B_{\nu}\big\}\Big)\ge e^{\lambda^2 a_{\lambda}(h(\nu)-\eps\big)}.$$
		\item[(ii)] for any  neighborhood  $B_{\rho}$   and  $\eps>0,$  we  have
		$$Card\Big(\big\{z^{\lambda}\in\skrig_p\,|\, U_2^{\lambda}\in B_{\nu}\big\}\Big)\le e^{\lambda^2 a_{\lambda}(h(\nu)+\eps\big)},$$
	\end{itemize}
\end{cor}

where $Card(A)$  means  the  cardinality  of  $A.$

\begin{remark}
	For $\nu=q\pi\otimes\pi,$  we  have	$\displaystyle Card\Big(\Big\{y\in\skrig_p\,\Big\}\Big)\approx e^{\lambda^2\,a_{\lambda}\|q\pi\otimes\pi\|h\big(q\pi\otimes\pi/\|q\pi\otimes\pi\|\big)}.$
	
\end{remark}

\begin{cor}\label{randomg.LDM}\label{main2d}
	Let   $Z^{\lambda}$  is  a super critical powered Sinr  network  with rate measure
$\lambda \eta:D \to [0,1]$ and   a   power  probability  function  $\skrik(\eta)=ce^{-c\eta}, \eta>0$   and  path  loss  function   $\beta(r)=r^{-\ell}, $  for  $\ell>0.$  Thus, the connectivity  probability  $Q^{z^\lambda} $  of   $Z^{\lambda}$  satisfies  $a_{\lambda}^{-1}Q^{z^\lambda}\to q$  and  $\lambda a_{\lambda} \to \infty.$ 
	\begin{itemize}
		\item  Let  $F$  be  closed  subset  $\skrim_{\pi}$.  Then  we  have  
		$$\limsup_{\lambda\to\infty}\frac{1}{\lambda^2 a_{\lambda}}\log \P_{\pi}\Big\{Z^{\lambda}\in \skrig_P\,\Big|\, U_2 ^{\lambda}\in F\Big\}\le -\sfrac{1}{2}\inf_{\pi\in F}\Big\{\skrih(\nu\|q\pi\otimes\pi)\Big\}.$$
		\item  Let  $O$  be  open  subset  $\skrim_{p}$.  Then  we  have  
		$$\liminf_{\lambda\to\infty}\frac{1}{\lambda^2 a_{\lambda}}\log \P_{\pi}\Big\{Z^{\lambda}\in \skrig_p\,\Big|\, U_2^{\lambda}\in O\Big\}\ge -\sfrac{1}{2}\inf_{\nu\in O}\Big\{\skrih(\nu\|q\pi\otimes\pi)\Big\}.$$
	\end{itemize}
	
\end{cor}

\section{ Proof  of  Theorem~\ref{main1a} by  Gartner-Ellis  Theorem  and  Method of  Mixtures }\label{proofmain1b}\label{Sec3}

\subsection{Proof  of   Theorem~\ref{main1a-1}(i)}\label{proofmain}

Suppose   $A_1,...,A_n$   is  a  decomposition  of  the  space  $\skrid\times \R_{+}.$   Observe  that,  for  every $(x,y)\in A_i\times A_j,\, i,j=1,2,3,...,n,$  $\lambda U_2^{\lambda}(x,y)$  given  $\lambda U_1^{\lambda}(x)=\lambda\mu(x)$  is  binomial  with  parameters  $\lambda^2\mu(x)\mu(y)/2$  and  $Q^{z^{\lambda}}(x,y).$ Let  $q$  be  the  exponential  distribution  with  parameter $c.$  We  recall  the  function  $q_{\lambda}^{\skrid}  $  from  the  previous  sections  and  note  that
Lemma~\ref{main1c}  is  key  component    in  the  application  of  the  Gartner-Ellis  Theorem.  See \cite{DZ1998}.  
\begin{lemma}\label{randomg.LDM1b}\label{main1ca}
	
		Let   $Z^{\lambda}$  is  a super critical powered Sinr  network  with rate measure
	$\lambda \mu:D \to [0,1]$ and   a   power  probability  function  $\skrik(\eta)=ce^{-c\eta}, \eta>0$   and  path  loss  function   $\beta(r)=r^{-\ell}, $  for  $\ell>0.$  Thus, the connectivity  probability  $Q^{z^\lambda} $  of   $Z^{\lambda}$  satisfies  $a_{\lambda}^{-1}Q^{z^\lambda}\to q$  and  $\lambda a_{\lambda} \to \infty.$
	Let   $Z^{\lambda}$  be  a supercritical SINR network,  conditional  on the  event  $U_1^{\lambda}=\pi.$  Let $ g:\skriw\times \skriw\to \R$ be  bounded  function.  Then,
	
	$$\begin{aligned}\lim_{\lambda\to\infty}\frac{1}{\lambda}\log\me \Big\{e^{\lambda\langle g, \, U_2^{\lambda}\rangle }\Big | U_1^{\lambda}=\pi\Big\}&=\frac{1}{2}\lim_{n\to\infty}\sum_{j=1}^{n}\sum_{i=1}^{n}\Big\langle g,\, q\pi\otimes\pi\Big\rangle _{A_i \times A_j}\\
	&=\frac{1}{2}\Big\langle g,\, q\pi\otimes\pi\Big\rangle _{\skriw \times \skriw}.
	\end{aligned}$$
\end{lemma}	
\begin{Proof}
 Now  we  observe  that 
	$$ \me\Big \{ e^{\int \int \lambda g(x,y)U_2^{\lambda}(dx,dy)/2} \Big |U_1^{\lambda}=\pi\Big \}= \me\Big\{\prod_{x \in \skriw} \prod_{y \in \skriw} e^{\lambda g(x,y)U_2^{\lambda}(dx,dy)/2} \Big\}$$
	
	$$ \me\Big \{\prod_{x \in \skriw} \prod_{y \in \skriw} e^{g(x,y)\lambda U_2^{\lambda}(dx,dy/2)}\Big\} = \prod_{i=1} \prod_{j=1} \prod_{x \in A_i} \prod_{y \in A_j} \me\Big\{e^{g(x,y) \lambda U_2^{\lambda}(dx,dy)/2  }\Big \} $$
	
	$$ \log \Big \{e^{\lambda \langle g,U_2^{\lambda}\rangle/2}\Big|U_1^{\lambda}=\pi\Big\} = \sum_{j=1}^{n} \sum_{i=1}^{n}\int_{B_j}\int_{B_i}\log\Big[1-Q^{z^{\lambda}}(x,y)+Q^{z^{\lambda}}(x,y)e^{g(x,y)/\lambda a_{\lambda}}\Big]^{\lambda^{2} \pi\otimes\pi(dx,dy)/2}+o(n) $$

	By the  dominated convergence theorem 
	$$   \frac{1}{\lambda} \log E \{e^{\lambda \langle g,U_2^{\lambda}\rangle/2 } \mid U_1^{\lambda}=\pi\} = \frac{1}{\lambda}\sum_{j=1} \sum_{i=1} \int_{A_i} \int_{A_j} \log\Big [1-\big(1-e^{g(x,y)/\lambda a_{\lambda}}) Q^{z^{\lambda}}(x,y)  \Big]^{ \lambda^2 \pi\otimes\pi(dx,dy)/2} +o(n)/\lambda$$
	$$  \frac{1}{\lambda} \log \me \{e^{\lambda \langle g,U_2^{\lambda}\rangle /2} \big| U_1^{\lambda}=\pi\} = \lim_{\lambda \rightarrow \infty } \sum_{j=1} \sum_{i=1}\int_{A_i} \int_{A_j} \log \Big[1+g(x,y) q(x,y)/\lambda +o(\lambda)/\lambda   \Big]^{ \lambda \pi\otimes\pi(dx,dy)/2}+o(n)/\lambda $$

	$$\lim_{\lambda \rightarrow \infty}\frac{1}{\lambda} \log \me\{e^{\lambda \langle g,U_2^{\lambda}\rangle/2} \mid U_1^{\lambda}=\pi\} 
	=\frac{1}{2}\sum_{j=1}^{n}\sum_{i=1}^{n} \Big\langle g,\, q\pi\otimes\pi\Big\rangle _{A_i \times A_j} $$
	
	$$ \begin{aligned}
	 \lim_{\lambda \rightarrow \infty} \frac{1}{\lambda} \log \me \{e^{\lambda \langle g,U_2^{\lambda}\rangle/2 } \Big |U_1^{\lambda}=\pi\}& =\frac{1}{2} \lim_{n \rightarrow \infty } \sum_{j=1}^{n}\sum_{i=1}^{n} \Big\langle g,\, q\pi\otimes\pi\Big\rangle _{A_i \times A_j}\\
	&	=\frac{1}{2} \Big\langle g,\, q\pi\otimes\pi\Big\rangle _{\skriw \times \skriw} . 
\end{aligned}$$

Hence,	by Gartner-Ellis theorem, conditional  on the  event $\Big\{U_{1}^{\lambda}= \mu\Big\}$, $U_2^{\lambda}$ obey a  large  deviation  principle with speed $\lambda$  and variational formulation of  the  rate function
		$$ I_{\mu}(\pi) = \frac{1}{2}\sup_{g} \Big\{  \Big\langle g,\, \pi\Big\rangle _{\skriw \times \skriw}-  \Big\langle g,\, q\pi\otimes\pi\Big\rangle _{\skriw\times \skriw}\Big\}$$ 
	
	 which   when solved,  see  example   \cite{DA2006},	would clearly   reduces  to  the good rate function  given by 	$$	I_{\pi}^1(\nu)=0.$$

\end{Proof}

\subsection{Proof  of   Theorem~\ref{main1a-1}(ii) }\label{proofmain}

Similarly we take   $A_1,...,A_n$  a  a  decomposition  of  the  space  $\skrid\times \R_{+}.$   We  recall  the  function  $h_{\lambda}^{D}  $  from  the  previous  sections   and  state   the  following  Lemma. Lemma~\ref{main1da}  is  key  component    in  the  application  of  the  Gartner-Ellis  Theorem.  See,  \cite{DZ1998}.  
\begin{lemma}\label{randomg.LDM1b}\label{main1da}	Let   $Z^{\lambda}$  is  a super critical powered Sinr  network  with rate measure
$\lambda \mu:D \to [0,1]$ and   a   power  probability  function  $\skrik(\eta)=ce^{-c\eta}, \eta>0$   and  path  loss  function   $\beta(r)=r^{-\ell}, $  for  $\ell>0.$  Thus, the connectivity  probability  $Q^{z^\lambda} $  of   $Z^{\lambda}$  satisfies  $a_{\lambda}^{-1}Q^{z^\lambda}\to q$  and  $\lambda a_{\lambda} \to \infty.$
Let   $Z^{\lambda}$  be  a supercritical SINR network,  conditional  on the  event  $U_1^{\lambda}=\pi.$  Let $ g:\skriw\times \skriw\to \R$ be  bounded  function.  Then,
	
	$$\begin{aligned}\lim_{\lambda\to\infty}\frac{1}{\lambda^2 a_{\lambda}}\log\me \Big\{e^{\lambda^2 a_{\lambda}\langle g, \, U_2^{\lambda}\rangle }\Big | U_1^{\lambda}=\pi\Big\}&=-\frac{1}{2}\lim_{n\to\infty}\sum_{j=1}^{n}\sum_{i=1}^{n}\Big\langle 1-e^{g},\, q\pi\otimes\pi\Big\rangle _{A_i \times A_j}\\
	&=-\frac{1}{2}\Big\langle 1-e^{g},\, q\pi\otimes\pi\Big\rangle _{\skriw \times \skriw}.
	\end{aligned}$$
\end{lemma}	
\begin{Proof}
	Now  we  observe  that 
	$$ \me\Big \{ e^{\int \int \lambda^2 a_{\lambda} g(x,y)U_2^{\lambda}(dx,dy)/2} \Big |U_1^{\lambda}=\pi\Big \}= \me\Big\{\prod_{x \in \skriw} \prod_{y \in \skriw} e^{\lambda^2 a_{\lambda} g(x,y)U_2^{\lambda}(dx,dy)/2} \Big\}$$
	
	$$ \me\Big \{\prod_{x \in \skriw} \prod_{y \in \skriw} e^{g(x,y)\lambda U_2^{\lambda}(dx,dy/2)} = \prod_{i=1} \prod_{j=1} \prod_{x \in A_i} \prod_{y \in A_j} \me\Big\{e^{\lambda^2 a_{\lambda}g(x,y)  U_2^{\lambda}(dx,dy)/2  }\Big \}\times e^{o(n)} $$
	
	$$ \log \Big \{e^{\lambda^2 a_{\lambda}\langle g,U_2^{\lambda}\rangle/2}\Big|U_1^{\lambda}=\pi\Big\} = \sum_{j=1}^{n} \sum_{i=1}^{n}\int_{A_j}\int_{A_i}\log\Big[1-Q^{z^{\lambda}}(x,y) )+Q^{z^{\lambda}}(x,y) e^{g(x,y)}\Big]^{\lambda^{2} \pi\otimes\pi(dx,dy)/2}+o(n) $$

	By the  dominated convergence theorem 
	$$   \frac{1}{\lambda^2 a_{\lambda}} \log E \{e^{\lambda \langle g,U_2^{\lambda}\rangle/2 } \mid U_1^{\lambda}=\pi\} = \frac{1}{\lambda^2 a_{\lambda}}\sum_{j=1} \sum_{i=1} \int_{A_i} \int_{A_j} \log\Big [1-\big(1-e^{g(x,y)}) Q^{z^{\lambda}}(x,y)  \Big]^{ \lambda^2 \pi\otimes\pi(dx,dy)/2} +o(n)/\lambda^2a_{\lambda}$$
	$$  \frac{1}{\lambda^2 a_{\lambda}} \log \me \{e^{\lambda \langle g,U_2^{\lambda}\rangle /2} \| U_1^{\lambda}=\pi\} = \lim_{\lambda \rightarrow \infty } \sum_{j=1} \sum_{i=1}\int_{A_i} \int_{A_j} \log \Big[1-(1-e^{g(x,y)}) Q^{z^{\lambda}}(x,y)  \Big]^{ \lambda \pi\otimes\pi(dx,dy)/2} +o(n)/\lambda^2a_{\lambda}$$
	
	$$\lim_{\lambda \rightarrow \infty} \frac{1}{\lambda^2 a_{\lambda}} \log \me \Big\{e^{\lambda \langle g,U_2^{\lambda}\rangle /2} \big| U_1^{\lambda}=\pi\Big \} =- \frac{1}{2} \sum_{j=1} \sum_{i=1} \int_{A_i} \int_{A_j}\Big [(1-e^{g(x,y)})q(x,y)\pi\otimes\pi(dx,dy)\Big] $$
	
	$$\lim_{\lambda \rightarrow \infty}\frac{1}{\lambda^2 a_{\lambda}} \log \me\{e^{\lambda \langle g,U_2^{\lambda}\rangle/2} \mid U_1^{\lambda}=\pi\} 
	=- \frac{1}{2}\sum_{j=1}^{n}\sum_{i=1}^{n} \Big\langle 1-e^{g},\, q\pi\otimes\pi\Big\rangle _{A_i \times A_j} $$
	
	$$ \begin{aligned}
	\lim_{\lambda \rightarrow \infty} \frac{1}{\lambda^2 a_{\lambda}} \log \me \{e^{\lambda \langle g,U_2^{\lambda}\rangle/2 } \Big |U_1^{\lambda}=\pi\}& =-\frac{1}{2} \lim_{n \rightarrow \infty } \sum_{j=1}^{n}\sum_{i=1}^{n} \Big\langle 1-e^{g},\, q\pi\otimes\pi\Big\rangle _{A_i \times A_j}\\
	&	=-\frac{1}{2} \Big\langle 1-e^{g},\, q\pi\otimes\pi\Big\rangle _{\skriw \times \skriw}  
	\end{aligned}$$

	Hence,	by Gartner-Ellis theorem, conditional  on the  event $\Big\{M_{1}^{\lambda}= \mu\Big\}$, $U_2^{\lambda}$ obey a  large  deviation  principle with speed $\lambda$  and variational  formulation of  the  rate function
	$$ I_{\mu}(\pi) = \frac{1}{2}\sup_{g} \Big\{  \Big\langle g,\, \pi\Big\rangle _{\skriw \times \skriw}+  \Big\langle 1-e^{g},\, q\pi\otimes\pi\Big\rangle _{\skriw\times \skriw}\Big\}$$ 
	
	which   when solved,  see  example   \cite{DA2006},	would clearly   reduces  to  the good rate function  given by 
	\begin{equation}
	I_{\pi}^2(\nu)=  \frac{1}{2}\skrih(\nu\|q \pi\otimes\pi).
	\end{equation}
	
\end{Proof}

\subsection{ Proof of  Theorem~\ref{main1a}   by  Method  of  Mixtures.}\label{Sec4}For any $\lambda\in (0,\infty)$ we define
$$\begin{aligned}
\skrim_{\lambda}(\skriw) & := \Big\{ \mu\in \skrim(\skriw) \, : \, \lambda\mu(x) \in \N \mbox{ for all } x\in \skriw\Big\},\\
\tilde \skrim_{\lambda }(\skriw\times \skriw) & := \Big\{ \pi\in
\tilde\skrim(\skriw\times \skriw) \, : \, 
\lambda \,\pi(x,y) \in \N,\,  \mbox{ for all } \, x,y\in \skriw
\Big\}\, .
\end{aligned}$$

We denote by
$\Theta_{\lambda}:=\skrim_{\lambda }(\skriw)$
and
$\Theta:=\skrim(\skriw)$.
We  write 
$$\begin{aligned}
P_{ \mu_{\lambda}}^{(\lambda)}(\eta_{\lambda}) & := \prob\big\{U_2^{\lambda}=\eta_{\lambda} \, \big| \, U_1^{\lambda}=\pi_{\lambda}\big\}\, ,\\
P^{(\lambda)}(\mu_{\lambda}) & :=
\prob\big\{U_1^{\lambda}=\pi_{\lambda}\big\}
\end{aligned}$$

Th joint distribution of $U_1^{\lambda}$ and $U_2^{\lambda}$ is
the mixture of $P_{ \mu_{\lambda}}^{(\lambda)}$ with
$P^{(\lambda)}(\mu_{\lambda}),$    as follows: 
\begin{equation}\label{randomg.mixture}
d\tilde{P}^{\lambda}( \mu_{\lambda}, \eta_{\lambda}):= dP_{	\mu_n}^{(\lambda)}(\eta_{\lambda})\, dP^{(\lambda)}( \mu_{\lambda}).\,
\end{equation}

(Biggins, Theorem 5(b), 2004) gives criteria for the validity of
large deviation principles for the mixtures and for the goodness of
the rate function if individual large deviation principles are
known. The following three lemmas ensure validity of these
conditions.

Observe   that the  family of
measures $({P}^{(\lambda)} \colon \lambda\in(0,\infty))$  is  exponentially tight on
$\Theta.$

\begin{lemma}[] \label{Com4}

	\begin{itemize}
	
\item[(i)] 	The  family of
	measures $(\tilde{P}^{\lambda} \colon \lambda\in(0,\infty))$  is  exponentially tight on
	$\Theta\times\tilde\skrim_*(\skriw\times \skriw).$

\item[(ii)] The  family	measures $(Q^{z^{\lambda}} \colon \lambda\in(0,\infty))$  is  exponentially tight on
	$\Theta\times\tilde\skrim_*(\skriw\times \skriw).$
\end{itemize}
\end{lemma}

We  refer  to  \cite[Lemma~4.3]{SAD2020} for  similar  proof  for  Large  deviation  Principle on  the  scale  $\lambda^2$

Define the function
$I_{Sc}^2,I_{Sc}^1\colon{\Theta}\times\skrim_*(\skriw\times \skriw)\rightarrow[0,\infty],$  by

	\begin{equation}
\begin{aligned}
I_{Sc}^{1}\big(\pi,\nu\big)= \left\{\begin{array}{ll}H\Big(\pi\Big |\eta\otimes q\Big)&\,\,\mbox{ if $ \nu=q\pi\otimes\pi $ }\\
\infty & \mbox{otherwise.}
\end{array}\right.
\end{aligned}
\end{equation}

\begin{equation}
I_{Sc}^{2}\big(\pi,\nu\big)= \frac{1}{2} \skrih\Big(\nu\|q\pi\otimes\pi\Big).
\end{equation}

\begin{lemma}[]\label{Com5}
\begin{itemize}
	
\item[(i)]	$I_{Sc}^1$ is lower semi-continuous.
\item  [(ii)] $I_{Sc}^2$  is  lower  semi-continuous.
\end{itemize}
\end{lemma}

By (Biggins, Theorem~5(b), 2004) the two previous lemmas,  the  LDP  for  the  empirical  power  measure, see,  \cite[Theorem~2.1]{SAD2020} and the
large deviation principles we have established
Theorem~\ref{main1a-1}  ensure
that under $(\tilde{P}^{\lambda})$    and   $Q^{z^{\lambda}}$ the random variables $(\pi_{\lambda}, \eta_{\lambda})$   satisfy a large deviation principle on
$\skrim(\skriw) \times \tilde\skrim(\skriw\times \skriw)$ and   	$\Theta\times\tilde\skrim_{\lambda}(\skriw\times \skriw)$  on  the  speeds  $\lambda$ and  $\lambda^2 a_{\lambda}$ with good rate functions  $I_{Sc}^1$   and $I_{Sc}^2$  respectively,  which  ends  the  proof of  Theorem~\ref{main1a}.

\section{Proof of   Theorem~\ref{main1b} by Large  deviations }

In  order  to establish    the  asymptotic equipartition property,  we  first  prove   a  weak  law  of large  numbers  for   the empirical  powere  measure  and  the  empirical  connectivity  measure  of  the  SINR  network.
\begin{lemma}\label{WLLN}	Let   $Z^{\lambda}$  is  a super critical powered Sinr  network  with rate measure
$\lambda \mu:D \to [0,1]$ and   a   power  probability  function  $\skrik(\eta)=ce^{-c\eta}, \eta>0$   and  path  loss  function   $\beta(r)=r^{-\ell}, $  for  $\ell>0.$  Thus, the connectivity  probability  $Q^{z^\lambda} $  of   $Z^{\lambda}$  satisfies  $a_{\lambda}^{-1}Q^{z^\lambda}\to q$  and  $\lambda a_{\lambda} \to \infty.$
	Then,	$$\lim_{\lambda\to\infty}\P\Big\{\sup_{(x,\eta_x)\in\skriw}\Big|L_1^{\lambda}(x,\eta_x)-\mu\otimes \skrik(x,\eta_x) \Big|>\eps\Big\}=0$$ and  
	$$\lim_{\lambda\to\infty}\P\Big\{\sup_{([x,\eta_x],[y,\eta_y])\in\skriw\times\skriw}\Big|L_2^{\lambda}([x,\eta_x],[y,\eta_y])-q\mu\otimes \skrik\times \mu\otimes \skrik([x,\mu_x],[y,\mu_y]) \Big|>\eps\Big\}=0$$
\end{lemma}

\begin{proof}
	Let
	$$ F_{1,\skriw}=\Big\{\pi:\sup_{(x,\eta_x)\in\skriw}|\pi(x,\eta_x)-m\otimes \skrik(x,\eta_x)|>\eps\Big\},$$ $$F_{2,\skriw}=\Big\{\nu:\sup_{([x,\eta_x],[y,\eta_y])\in\skriw\times\skriw}|\nu([x,\eta_x],[y,\eta_y])-q\mu\otimes \skrik\times \mu\otimes \skrik([x,\eta_x],[y,\eta_y])|>\eps\Big \}$$
	
	and  $F_{3,\skriw}=F_{1,\skriw}\cup F_{2,\skriw}.$     Now, observe  from  Theorem~\ref{main1a}  that
	
	$$\lim_{\lambda\to\infty}\frac{1}{\lambda}\log \P\Big\{(L_1^{\lambda},L_2^{\lambda})\in F_{3,\skriw}^{c}\Big \}\le -\inf_{(\pi,\varpi) \in  F_{3,\skriw}^{c}}I(\pi,\varpi).$$
	
	It  suffices  for  us  to  show  that   $I$  is strictly  positive.  Suppose  there  is  a  sequence  $(\pi_n,\varpi_n)\to(\pi,\varpi)$   such  that  $I(\pi_{\lambda},\varpi_{\lambda})\downarrow I(\pi,\varpi)=0.$  This  implies  $\pi=\mu\otimes \skrik$  and  $\varpi=q\mu\otimes \skrik\times \mu\otimes \skrik$  which  contradicts  $(\pi,\varpi)\in F_3^{c}.$  This  ends  the  proof of  the  Lemma.
\end{proof} 

Now,  the  distribution  of  the  marked  PPP  $P(x)=\P\Big\{X^{\lambda}=x\Big\}$  is  given  by  $$P_{\lambda}(x)=\prod_{i=1}^{I}|\mu\otimes \skrik(x_i,\eta_i)\prod_{(i,j)\in E}\frac{Q^{z^\lambda}([x_i,\mu_i],[y_j,\mu_j])}{1-Q^{z^\lambda}([x_i,\mu_i],[y_j,\mu_j])}\prod_{(i,j)\in \skrie} (1-Q^{z^\lambda}([x_i,\mu_i],[y_j,\mu_j]))\prod_{i=1}^{I}(1-Q^{z^\lambda}([x_i,\mu_i],[y_j,\mu_j]))$$

$$\begin{aligned}-\frac{1}{a_{\lambda}\lambda^2\log \lambda}\log P_{\lambda}(x)&=\frac{1}{a_{\lambda}\lambda\log \lambda}\Big\langle -\log \mu\otimes Q\,,L_1^{\lambda}\Big\rangle +\frac{1}{\log \lambda}\Big \langle -\log \Big(\sfrac{Q^{z^\lambda}}{1-Q^{z^\lambda}}\Big ) \,,L_{2}^{\lambda}\Big \rangle\\
 & +\frac{1}{a_{\lambda}\log \lambda}\Big\langle-\log (1-Q^{z^\lambda})\,,L_1^{\lambda}\otimes L_{1}^{\lambda}\Big\rangle +\frac{1}{a_{\lambda}\lambda\log\lambda}\Big\langle-\log (1-Q^{z^\lambda})\,,L_{\Delta}^{\lambda}\Big\rangle
\end{aligned}  $$

Notice,        $$\displaystyle \lim_{\lambda\to\infty}\frac{1}{a_{\lambda}\lambda\log \lambda}\Big\langle -\log \mu\otimes \skrik\,,L_1^{\lambda}\Big\rangle=\lim_{\lambda\to \infty}\frac{1}{\lambda}\Big\langle-\log (1-Q^{z^{\lambda}}\,,L_{\Delta}^{\lambda}\Big\rangle=\lim_{\lambda\to\infty}\frac{1}{a_{\lambda}\log \lambda}\Big\langle-\log (1-Q^{z^\lambda})\,,L_1^{\lambda}\otimes L_{1}^{\lambda}\Big\rangle=0.$$

Using,  Lemma~\ref{WLLN}  we  have  

 $$\lim_{\lambda\to\infty}\frac{1}{\log \lambda}\Big \langle -\log \Big(Q^{z^{\lambda}}/(1-Q^{z^{\lambda}}\Big) \,,L_{2}^{\lambda}\Big \rangle=\Big \langle \1 \,, q\mu\otimes \skrik\times \mu\otimes \skrik\Big \rangle$$

which  concludes  the  proof  of  Theorem~\ref{main1b}.

\section{Proof  of  Theorem~\ref{main1c}, Corollary~\ref{cardinality}, Corollary~\ref{main2d}}\label{Sec5}
For  $\pi\in\skrim(\skriw)$  we  define  the  spectral  potential  of  the  marked  SINR graph $(Z^{\lambda})$  conditional  on  the  event  $\big\{L_1^{\lambda}=\pi \big\},$   $\rho_q(g,\pi) $  as

\begin{equation}\label{LLDP.equ1}
\phi_{q}(g,\pi)= \Big \langle -(1-e^g)\,,\,q\pi\otimes \pi \Big \rangle.
\end{equation}

Note that  remarkable  properties  of  a  spectral  potential,  see   or \cite{SAD2020}  holds  for  $\phi_{q} $.

For $\pi\in\skrim(\skriw\times\skriw)$, we  observe that  $I_{\pi}( \pi)$ is the  Kullback  action  of  the marked  SINR graph $Z^{\lambda}$.

\begin{lemma}\label {LLDP.equ2} The  following  hold  for the  Kullback  action or divergence function $I_{\pi}(\pi)$:
	\begin{itemize}
		\item $$I_{Sc}(\pi)=\sup_{g\in \skric}\big\{\langle g,\, \pi\rangle-\phi_{q}(g,\pi) \big\}$$ 
		\item  The  function  $I_{Sc}(\pi) $  is  convex  and  lower semi-continuous   on  the  space  $\skrim(\skriw\times\skriw).$
		\item For  any real  $\alpha$,  the  set  $\Big\{ \pi\in \skrim(\skriw\times\skriw): \, I_{Sc}(\pi)\le \alpha \Big\}$  is  weakly  compact.
	\end{itemize}
\end{lemma}
The  proof  of  Lemma~\ref{LLDP.equ2} is  omitted from  the  article. Interested  readers  may  refer  to  \cite{DA2017a}  for  similar proof for empirical  measures of ` the  Typed  Random  Graph  Processes  or  See,  for  example  \cite{DA17b} for  the  multitype Galton-Watson  processes  and/or the  reference therein, \cite{BIV15},  for proof of  the  lemma for  empirical measures  on  measurable  spaces.\\

Note  from  Lemma~\ref{LLDP.equ2}  that,  for  any  $\eps>0$, there  exists some  function  $g\in\skriw\times\skriw$   such  that  
$$I_{Sc}(\pi)-\sfrac{\eps}{2}< \langle g\,,\, \pi \rangle -\phi_{q}(g,\pi).$$

We define  the  probability  distribution  of  the  powered $Z$   by $P_{\pi} $  by  

$$P_{\pi}(z)=\prod_{(i,j)\in E}e^{g(x_i,x_j)}\prod_{(i,j)\in \skrie}e^{h_\lambda(x_i,x_j)}	,$$

where  $$h_{\lambda}(x,y)=\frac{1}{ a_{\lambda}}\log\Big[1-Q^{z^{\lambda}}(x,y)+Q^{z^{\lambda}}(x,y)e^{g(x,y)}\Big]$$
Then,  observe  that  

$$\begin{aligned}
\frac{dP_{\pi}}{d\tilde{P}_{\pi}}(z)&=\prod_{(i,j)\in E}e^{-g(x_i,x_j)}\prod_{(i,j)\in \skrie} e^{-h_{\lambda}(x_i,x_j) a_{\lambda}}\\
&= e^{-\lambda^2 a_{\lambda} (\langle\sfrac{1}{2} g,L_2^{\lambda}\rangle-\lambda^2 a_{\lambda} \langle \sfrac{1}{2}h_{\lambda},L_1^{\lambda}\otimes L_1^{\lambda}\rangle )+\langle \sfrac{1}{2} h_{\lambda}, L_{\Delta}^{\lambda} \rangle }
\end{aligned}$$

Now  define  the  neighbourhood  of  $\nu,$    $B_{\nu}$ by
$$B_{\nu}:=\Big\{\omega\in \skrim(\skriw\times\skriw): \, \langle g, \omega \rangle-\rho_q(g,\pi) >\langle g, \nu\rangle -\rho_q(g,\pi)-\eps/2 \Big \}$$

Note  that  under the  condition  $L_{2}^{\lambda}\in B_{\nu}$   we  have

$$\begin{aligned}
\frac{dP_{\pi}}{d\tilde{P}_{\pi}}(z)< e^{-\lambda^2 a_{\lambda} (\langle\sfrac{1}{2} g,L_2^{\lambda}\rangle-\lambda^2 a_{\lambda} \langle \sfrac{1}{2}h_{\lambda},L_1^{\lambda}\otimes L_1^{\lambda}\rangle )+\langle \sfrac{1}{2} h_{\lambda}, L_{\Delta}^{\lambda} \rangle }<e^{-\lambda^2 a_{\lambda}  I_{Sc}(\nu)+\lambda^2a_{\lambda} \eps} \end{aligned}$$

Therefore,  we  obtain 

$$P_{\pi }\Big\{Z^{\lambda }\in\skrig_P\Big | L_2^{\lambda}\in B_{\nu}\Big\}\le \int \1_{\{L_2^{\lambda}\in B_{\nu}\}} d\tilde{P}_{\pi}(z^{\lambda})(z)\le \int e^{-\lambda^2 a_{\lambda} I_{Sc}(\pi)-\lambda\eps}d\tilde{P}_{\pi}(z^{\lambda}) \le e^{-\lambda^{2}a_{\lambda} I_{Sc}(\nu)-\lambda^2 a_{\lambda}\eps}.$$.

Note that $I_{Sc} (\nu)=\infty$ implies  Theorem ~\ref{main1b} (ii),  hence  it  sufficient for  us  to deduce  that  the  result is  true  for  a probability  distribution  of  the form $\nu=e^{g}\pi\otimes\pi$  and  for  $I_{Sc}(\nu)=\sfrac{1}{2}\skrih(\nu\|q\pi\otimes\pi).$  Fix  any  number  $\eps>0$  and  any  neigbourhood $B_{\nu}\subset \skrim(\skriw\times\skriw)$.  Now  define  the  sequence  of  sets  $$\skrig_{p}^{\lambda}=\Big\{ y\in \skrig_{p}: L_{2}^{\lambda}(y)\in B_{\nu}\Big |\langle g, L_{2}^{\lambda}\rangle-\phi_{q}(g,\pi)\Big |\le \sfrac{\eps}{2}\Big\} .$$     

Note that  for  all  $y\in \skrig_p^{\lambda}$ we  have    

$$\begin{aligned}
\frac{dP_{\pi}}{d\tilde{P}_{\pi}}> e^{-\lambda^2 a_{\lambda}\langle\sfrac{1}{2} g,\nu\rangle+\lambda^2 a_{\lambda}\phi_{q}(g,\,\pi)+\lambda^{2} a_{\lambda}\sfrac{\eps}{2}}\end{aligned}.$$

This  yields  
$$P_{\pi}(\skrig_P^{\lambda}) =\int_{\skrig_P^{\lambda}}dP_{\pi}(y)\ge\int e^{-\lambda^2 a_{\lambda}\langle\sfrac{1}{2} g,\nu\rangle+\lambda^2 a_{\lambda}\phi_{q}(g,\,\pi)+\lambda^2 a_{\lambda} \sfrac{\eps}{2}}d\tilde{P}_{\pi}(y)\ge e^{-\lambda^2 a_{\lambda}\sfrac{1}{2}\skrih(\nu\|q\pi\otimes\pi)+\lambda^2 a_{\lambda} \eps}\tilde{P}_{\pi}(\skrig_P^{\lambda}).$$
Applying  the law  of  large  numbers, we  have  that  $\lim_{\lambda\to\infty}\tilde{P}_{\pi}(\skrig_P^{\lambda})=1.$  This  completes  of  the  Theorem.\\

{\bf  Proof of  Corollary~\ref{cardinality}}

The  proof  of  Corollary~\ref{cardinality}  follows  from  the  definition  of  the  Kullback action  and  Theorem~\ref{main1c} if  we  set   $\pi=\rho$  and  $\lambda\pi\otimes\pi(a,b)=\|\lambda\pi\otimes\pi\|,$  for  all  $(a,b)\in\skriy\times\skriy.$\\

{\bf  Proof of  Corollary~\ref{main2d}}

We observe that, by  Lemma~\ref{Com4} the  law  of   empirical connectivity measure is exponentially  tight. Henceforth, without loss of generality we can assume that the
set $F$  in Corollary~\ref{main2d}(ii) above is relatively compact. If we choose any $\eps> 0$; then for each functional  $\nu\in F$ we can find a weak neigbourhood such that the estimate of Theorem~\ref{main1c}(i) above holds. From
all these neigbhoourhood, we choose a finite cover of $\skrig_P$ and sum up over the estimate in Corollary~\ref{main2d}(i) above to obtain   	$$\limsup_{\lambda\to\infty}\frac{1}{\lambda}\log \P_{\pi}\Big\{Z^{\lambda}\in \skrig_P\,\Big|\, L_2 ^{\lambda}\in F\Big\}\le -\inf_{\nu\in F}I_{\pi}(\nu)+\eps.$$

As $\eps$  was arbitrarily chosen and the lower bound in Theorem~\ref{main2d}(ii)  implies the lower bound in
Theorem 2.2(i) we get  the desired results which completes the proof.

\section{Conclusion}\label{Sec6}
In  this  article  we   have  presented   a   joint  large  deviation  principle  for  the  empirical  power  measure  and  the  empirical connectivity  measure   of  telecommunication  networks in   the  $\tau-$  topology. From  this  large  deviation   principle  we  deduce  an  asymptotic equipartition  property  for  the  telecommunication  network  modelled  as  the  SINR  network  model.

We have  also  presented  a  Local  large  deviation  principle  for  the empirical connectivity  measure  given  the  empirical power  measure   and  from  this  result  we   had  deduce the  classical  MacMillian  theorm and  an  asymptotic   bound   for  the  set  alll posible  SINR  network  process. Finaly,  we the  authors had   also  presented  a  large  deviation  principle 
for  the  SINR  networks.  This  article  might may  be  regarded  as  a  first  step  in  the  proof  of  a  Lossy  asymptotic equipartition  property  for  the  SINR  networks. See,  \cite{DA17a}  and  \cite{DA17b}  for  similar  results  for  the  networked  data  structures modelled  as  colored  random  graph process and  for  the  hierarchical data  structure  modelled  as Galton-Watson tree process.  



\begin{thebibliography}{WWW98}
	\bibitem[1]{DZ1998}
	Dembo, A.  and  Zetouni, O.(1998).
	\newblock{ Large Deviations   Techniques and  applications }
	\newblock{Springers.}
	
	\bibitem[2]{DA2006}
	Doku-Amponsah,~K.(2006).
	\newblock{Large  Deviations and  Basic  Information  Theory  for  Hierarchical  and  Networked   Data  Structures}
	\newblock{Ph.D  Thesis},  Bath.
	
	\bibitem[3]{DM2010}
	Doku-Amponsah, K.  and Moeters. P. (2010).
	\newblock Large deviation principle  for  empirical  measures  of  coloured  random  graphs.
	\newblock {\emph { Ann. Appl.  Prob. 20(6),1089-2021.}}

	\bibitem[4]{JK2018}
	Jahnel, B. and Konig,W. (2003).
	\newblock Probabilstic  Methods  in  Telecommunication.
	\newblock  {\em Lecture  Notes.}  TU Berlin and  WiAS Berlin.
	


\bibitem[5]{BIV15}
{\sc Bakhtin~I.V..}
\newblock{Spectral Potential,  Kullback  Action,  and  Large  deviations  of  empirical measureson  measureable  spaces.}
\newblock{\emph{Theory  of  Probability  and  application. Vol.  50,No.4.(2015) pp.535-544.}}
\smallskip



\bibitem[6]{DA17a}
{\sc Doku-Amponsah,~K.}
\newblock{Lossy Asymptotic Equipartition property for Networked Data Structures.}
\newblock {\em Journal of  Mathematics and Statistics(JMSS) 13(2), pp.152-158 }
\smallskip

\bibitem[7]{DA17b}
{\sc Doku-Amponsah,~K.}
\newblock{Lossy Asymptotic Equipartition Property for hierarchical data structures.}
\newblock {\em  Far East  Journal  of  Mathematical  Sciences, Vol. 101, 2017, pp.1013-1024.}
\smallskip

\bibitem[8]{DA2012}
{\sc Doku-Amponsah,~K.}
\newblock{Asymptotic equipartition properties for hierarchical and networked  structures.}
\newblock ESAIM: PS 16 (2012): 114-138.DOI: 10.1051/ps/2010016.
\smallskip
\bibitem[9]{DA2017c}
{\sc Doku-Amponsah,~K.}
\newblock{Local Large Deviations, McMillian Theorem for multitype Galton-Watson Processes .}
\newblock \emph { Far East Journal of Mathematical Sciences, 2017, 102(10), pp. 2307-2319}.
\smallskip


\bibitem[10]{DA2017a}
Doku-Amponsah, K. (2017).
\newblock  Local Large deviation: A McMillian Theorem for Coloured Random Graph Processes
\newblock {\emph {  Journal of Mathematics and Statistics 13(4) (2017) 347-352 }.}
\smallskip

	\bibitem[11]{SAD2020}
	Sakyi-Yeboah, E., Asiedu, L.  and  Doku-Amponsah, K.(2020)
	\newblock{Local Large Deviation Principle, Large Deviation Principle and Information theory for the Signal -to- Interference and Noise Ratio Graph Models.} 
	\newblock{ To  appear  in \emph{ Journal  of  optimization and  information sciences}}
	
	\bibitem[12]{SKAD2020}
	Sakyi-Yeboah, E.,Kwofie,  C.,  Asiedu, L.  and  Doku-Amponsah, K.(2020)
	\newblock{Large  Deviation Principle  for Empirical Sinr  Measure of Critical  Telecommunication  Network.} 
	\newblock{ To  appear  in \emph{ To  appear  in Journal  of  optimization and  information sciences}}

	
	\bibitem[13]{We1995}
	Weiss, A. (1995).
	\newblock{ An introduction to large deviations for communication networks.}
	\newblock { IEEE Journal on Selected Areas in Communications,} 13(6), 938-952.
	
	
	
	
\end{thebibliography}
\end{document}